\documentclass{amsart}
\usepackage{graphicx,amssymb,amscd,graphics,verbatim,mathrsfs,latexsym,amsmath,amsthm,eucal}
\usepackage{latexsym}
\usepackage{color}
\usepackage{amsfonts}
\usepackage{amssymb, mathrsfs, amsfonts, amsmath}
\usepackage{amsbsy}
\usepackage{amsfonts}
\newtheorem{corollary}{Corollary}

\newtheorem{definition}{Definition}
\newtheorem{lemma}{Lemma}
\newtheorem{proposition}{Proposition}

\newtheorem{theorem}{Theorem}

\numberwithin{equation}{section}

\usepackage{color}
\usepackage{color}
\usepackage{color}
\usepackage{color}
\usepackage{color}

\begin{document}
\title[some characterizations on conformal vector fields]{Some characterizations of Riemannian manifolds endowed  with a conformal vector fields}
\author{A. Barros$^{1}$, I. Evangelista$^{2}$ and E.Viana$^3$}

\address[A. Barros]{Universidade Federal do Cear\'{a} - UFC, Departamento  de Matem\'atica, Campus do Pici, Av. Humberto Monte, Bloco 914,
	60455-760, Fortaleza / CE, Brazil.}
\email{abbarros@mat.ufc.br}

\address[I. Evangelista] {Universidade Federal do Delta do Parna\'iba- UFDPar, Curso de Matem\'atica, Av. S\~ao Sebasti\~ ao, 2819, 64202-020, Parna\'iba /PI, Brazil}
\email{evangelistaisrael@ufdpar.edu.br.}

\address[E. Viana]{Instituto Federal de Educa\c c\~ao, Ci\^encia e Tecnologia do Cear\'a - Campus Caucaia\\
	61609-090 Caucaia/CE, Brazil.}
\email{emanuel.mendonca@ifce.edu.br}

\subjclass[2010]{Primary: 53C20, 53A30}
\keywords{Conformal gradient vector fields; cosmic no-hair conjecture; manifolds with boundary; hemisphere of the Euclidean sphere}

\begin{abstract}
The aim of this article is to investigate the presence of a conformal vector $\xi$ with conformal factor $\rho$ on a compact Riemannian manifold $M$ with or without boundary $\partial M$.  We firstly prove that a compact Riemannian manifold $(M^n, g)\,,n \geq 3,$ with constant scalar curvature, with boundary $\partial M$ totally geodesic, in such way that the traceless Ricci curvature is zero in the  direction of $\nabla \rho,$ is isometric to a standard hemisphere. In the $4$-dimensional case, under the condition $\displaystyle\int_M|\mathring{Ric}|^2\langle \xi,\nabla \rho\rangle \,dM\leq0$, we show that, either $M$ is isometric to a standard sphere, or $M$ is isometric to a standard hemisphere. Finally, we give a partial answer for the cosmic no-hair conjecture.
\end{abstract}

\maketitle
\section{Introduction}
 
Conformal mappings and conformal vector fields were intensively studied during the last 150 years, in particular in Riemannian geometry of dimension $n\geq 3$. Namely, there have been notable findings concerning the characterization of an $n$-sphere $\mathbb{S}^n(c)$ as well as an hemisphere $\mathbb{S}^n_+(c),$ for some positive constant $c,$ utilizing a non-Killing conformal vector field, as documented in \cite{Deshmukh, deshmuk, deshmuk4, deshmuk5, deshmuk1,deshmuk3, ejiri, evviana} and \cite{evangelista}.

Before we proceed, given an $n$-dimensional Riemannian manifold $(M^n, g)$ with Levi–Civita connection $\nabla$, we recall that a conformal vector field $\xi$ on $M$ is said to be closed if the $1$-form $\xi^{\flat}$ is closed. This is easily seen to be equivalent to the existence of a smooth function $\rho:M\to \mathbb{R},$ called the conformal factor of $\xi,$ such that $$\nabla_{X}\xi=\rho X,$$ for all $X\in \mathfrak{X}(M)$. Whence we have  $$\rho=\dfrac{1}{n}{\rm div}\,\xi,$$ where ${\rm div}$ is the divergence operators on $M$. Henceforth, we also use the notation $\langle \,,\,\rangle$ for the metric $g$ or inner product induced by $g$ on tensor spaces.

It's important to note that closed vector fields are also referred to as concircular vector fields. These fields have surfaced in the examination of conformal mappings that preserve geodesic circles and boast intriguing applications in general relativity (cf. \cite{takeno}). The conformal nature of a vector field remains unchanged under a conformal alteration of the metric. However, a closed conformal vector field remains closed only if the conformal alteration maintains a constant conformal factor; in such instances, the factor of the field remains unaltered. Moreover, a closed and conformal vector field $X$ is labelled parallel if $\rho$ vanishes entirely. We highlight that numerous manifolds host non parallel closed; certain existence outcomes and explicit examples are detailed in \cite{caminha} and \cite{kr}.

After these preliminary remarks we may state our main results as follows.

\begin{theorem}\label{thm1}
Let $\xi$ be a non trivial conformal vector field with conformal factor $\rho$ on a compact Riemannian manifold $(M^n, g)$ with smooth totally geodesic boundary $\partial M$  of dimension $n \geq 3 $ and $\rho|_{\partial M}=0$. If the scalar curvature $R$ of $(M^{n}, g)$ is constant and $\mathring{Ric}(\nabla \rho)=0,$ then M is isometric to a hemisphere $\mathbb{S}^n_+(c)$ for some positive constant $c.$
\end{theorem}

Next, we present two rigidity results for the $4$-dimensional case.

\begin{theorem}\label{thm2}
	Let $(M^4,g)$ be a compact Riemannian manifold  with constant scalar curvature and $\xi$ be a non trivial conformal vector field on  $M^4$ with conformal factor $\rho$. If $\displaystyle\int_M|\mathring{Ric}|^2\langle \xi,\nabla \rho\rangle \,dM\leq0$, then $M^4$ is isometric to a standard sphere $\mathbb{S}^4(c)$ for some positive constant $c.$ 
\end{theorem}

\begin{theorem}\label{thm3}
	Let $(M^4,g)$ be a compact Riemannian manifold  with smooth totally geodesic boundary $\partial M$  and constant scalar curvature. Let $\xi$ be a non trivial conformal vector field on  $M^4$ with conformal factor $\rho$, which satisfies $\rho = 0$ on $\partial M$. If $\displaystyle\int_M|\mathring{Ric}|^2\langle \xi,\nabla \rho\rangle \,dM\leq0$, then $M^4$ is isometric to a hemisphere $\mathbb{S}^4_+(c)$ for some positive constant $c.$
\end{theorem}

In order to state our last result, we define  the static triple as follows:

\begin{definition}\label{DefStatic}
	A complete and connected Riemannian manifold $(M^{n}, g)$ with a (possibly non empty) boundary $\partial M$ is said to be static if there exists a non-negative function $f$ on $M$ satisfying
	\begin{equation}\label{static}
		(\Delta f)g+f Ric =\nabla^2 f,
	\end{equation}in $M \setminus\partial M,$ and $\partial M= f^{-1}(0)$, where $Ric$, $\Delta$ and $\nabla^2$ stand, respectively, for the Ricci tensor, the Laplacian operator and the Hessian form on $M^n$. 
	
	In this case, $(M^{n},\,g,\,f)$ is called a static triple or simply a static metric.
\end{definition}

We highlight that, in General Relativity, Equation \eqref{static} appears as static solutions of Einstein field equations. Furthermore, Corvino et al. [~\cite{coreich}, Proposition 2.1] showed that a static metric also has constant scalar curvature $R$. Boucher et al.~\cite{Bou} formulated a classic conjecture, called \textit{cosmic no-hair conjecture}, as follows:

\vspace{0,2cm}
\textit{The only $n$-dimensional compact static triple $(M^n,\,g,\,f)$ with positive scalar curvature and connected boundary $\Sigma$ is given by a round hemisphere $\mathbb{S}^{n}_{+}$, where the function $f$ is taken as the height function.}
\vspace{0,2cm}

Now  we give a partial answer for the cosmic no-hair conjecture. We point out that this result is similar to that one derived in \cite{Hwang-yun} for critical metrics.

\begin{theorem}\label{thm4}
Let $(M^n, g, \rho)$ be a simply connected, compact static metric with smooth boundary $\partial M$ totally geodesic and endowed with a smooth conformal vector field $\xi$ which conformal factor is $\rho.$   If $\psi:=n\rho^2-\langle \nabla \rho, \xi\rangle\neq 0$, then $(M^n,g)$  is isometric to  a hemisphere $\mathbb{S}^n_+(c)$ for some positive constant $c.$
\end{theorem}

This paper is organized as follows. In Section \ref{PKL} we review the necessary knowledge about conformal vector fields in Riemannian geometry, as well as properties of static metrics. In Section \ref{prooftheorems} we give the proofs of Theorems \ref{thm1}, \ref{thm2}, \ref{thm3} and \ref{thm4}.

\section{Preliminaries and Key Lemmas}\label{PKL}
In this section, we will present some basic results of static metrics, as well as we will present the basic concepts about conformal vector fields and basic results that will be used in the proofs of the main theorems of this work.

 Note that tracing \eqref{static} we arrive at
\begin{equation}\label{4}
\triangle f=-\frac{Rf}{n-1},
 \end{equation} 
 and, by using again \eqref{static}, it is not difficult to check that
\begin{equation}\label{5}
 f\mathring {Ric}=\mathring{\nabla}^2 f,
 \end{equation}
where $\mathring {T}$ stands for the traceless of the tensor $T.$

We recall that a smooth vector field $\xi\in  \mathfrak{X}(M)$ is said to be \textit{conformal} if
\begin{equation}\label{eqconformalfield}
	\mathcal{L}_{\xi}g=2\rho g
\end{equation}
for a smooth function $\rho$ on $M$, where $\mathcal{L}_{\xi}$ is the Lie derivative in the direction of $\xi$. The function $\rho$ is the conformal factor of $\xi$ (cf.~\cite{Besse}). If $\xi$ is the gradient of a smooth function  on $M$, then $\xi$ is said to be a \textit{conformal gradient vector field}. In this case,  $\xi$ is  also closed. We say that $\xi$ is a non trivial conformal vector field if it is a non-Killing conformal vector field. 

Remember from Equation \eqref{eqconformalfield} that a vector field $\xi$ on Riemannian manifold $(M^n,\,g)$ is called conformal if $\mathcal{L}_{\xi}g$ is a multiple of $g.$ As a straightforward consequence of Koszul's formula, we have the following identity for  $\xi$ on $M$,
\begin{equation*}
	2g(\nabla_X \xi,Y)=\mathcal{L}_{\xi}g(X,Y)+d\eta (X,Y), \,\,\,X,Y\in\mathfrak{X}(M),
\end{equation*}
where $\eta$ stands for the dual 1-form associated to $Z$, that is, $\eta(Y)=g( Z,Y )$. We note that we can define $\varphi$
the following skew symmetric (1,1)-tensor:
\begin{equation*}
	d\eta(X,Y)=2g(\varphi(X),Y),\,\,\,X,Y\in\mathfrak{X}(M).
\end{equation*}
Therefore, by using the above equations, we arrive at
\begin{equation}\label{eqconf2}
	\nabla_X\xi=\rho X+\varphi(X), \,\,\, X\in \mathfrak{X}(M).
\end{equation}
Note that $\varphi$ gives us an idea of how much of the field $\xi$ is not closed vector field. There are several papers involving closed conformal vector fields  (see, e.g.,~\cite{caminha, hicks, tanno}). 

Observe that we can identify $\varphi$ with a  skew symmetric $(0,2)$-tensor and $\xi$ with the tensor $\xi(Y)=g(\xi,Y), \,\,Y\in \mathfrak{X}(M)$, to rewrite \eqref{eqconf2} as follows
\begin{equation}\label{eqderxitens}
	\nabla \xi=\rho g+\varphi,
\end{equation}
 where $\varphi(X,Y):=g(\varphi(X),Y)$. Moreover, we adopt the following expression for the curvature tensor
\begin{equation*}
R(X,Y)Z=\nabla_X \nabla_Y Z-\nabla_Y \nabla_X Z-\nabla_{[X,Y]}Z.
\end{equation*} 
Furthermore, on a Riemannian manifold $(M^n,\,g)$, the Ricci operator $S$ is defined using the Ricci tensor $Ric$ (see~\cite{Besse}) by
\begin{equation}
Ric(X,Y)=g(SX,Y), \,\,\,X,Y\in\mathfrak{X}(M).
\end{equation}
Whence, one can use  Equation \eqref{eqconf2} to get
\begin{equation}
R(X,Y)\xi=X(\rho)Y-Y(\rho)X+(\nabla \varphi)(X,Y)-(\nabla \varphi)(Y,X),
\end{equation}
where $(\nabla \varphi)(X,Y)=\nabla_{X}\big( \varphi Y\big) -\varphi(\nabla_X Y).$

Using the above equation and the expression for the Ricci tensor
$$Ric(X,Y)=\sum_{i=1}^n g(R(e_i,X)Y,e_i),$$
where $\{e_1, \ldots, e_n\}$ is a local orthonormal frame, we obtain
\begin{equation}\label{ric1}
Ric(\xi,Y)=-(n-1)Y(\rho)-\sum_{i=1}^n g(Y,(\nabla \varphi)(e_i,e_i)),
\end{equation}
where we used the skew symmetry of the operator $\varphi$. The above relation gives
\begin{equation}\label{ric2}
S(\xi)=-(n-1)\nabla \rho-\sum_{i=1}^n (\nabla \varphi)(e_i,e_i).
\end{equation}

Now, using again, Equation \eqref{eqconf2}, we compute the action of the rough Laplace operator $\Delta$ on
the vector field $\xi$ and find
\begin{equation}\label{rough laplace}
\Delta\xi=\nabla \rho+\sum_{i=1}^n(\nabla \varphi)(e_i,e_i).
\end{equation}

The following lemma, obtained previously in [~\cite{evviana}, Lemma 1], will be useful.

\begin{lemma}\label{lemmadiv}Let $(M^n,\,g)$ be a smooth compact Riemannian manifold with smooth totally geodesic boundary $\partial M$ and $\xi$ a smooth conformal vector field on $M$ with conformal factor $\rho$ satisfying $\rho|_{\partial M}=0$. Denote by ${\rm div}$ and ${\rm div}_{\partial M}$ the divergence operators on $M$ and $\partial M$, respectively, and
	by $\xi^T$ the tangential part of $\xi$ on $\partial M$. Then,
	\begin{equation}\label{eqdiv}
	{\rm div}(\xi)=n\rho,\,\,\,\,\,\,\,\ {\rm div}_{\partial M}(\xi^T)=(n-1)\rho.
	\end{equation}
	Furthermore,
	\begin{equation}\label{eqint}
	\int_M g(\nabla \rho,\xi)\,dM=-n\int_M \rho^2\,dM.
	\end{equation}
\end{lemma}

We present below properties that are valid, in general,  for conformal vector fields.

\begin{lemma}\label{yanod}
	Let $\xi$ be a non trivial conformal vector field on a Riemannian
	manifold $(M^n, g)$ with conformal factor $\rho$. Then
		\begin{enumerate}
			\item [(i)] \label{lemayano} $\mathcal{L}_{\xi}  \mathring{Ric}+(n-2)\mathring{\nabla^2} \rho=0.$
			\item [(ii)] ${\rm div}(\mathcal{L}_{\xi} \mathring{Ric})=-(n-2)\mathring{Ric}(\nabla \rho)-\dfrac{(n-1)(n-2)}{n}\nabla \Delta \rho-\dfrac{(n-2)R}{n}\nabla \rho.$ When the sacalar curvature $R$ is constant, we have
			\begin{equation}\label{mainequation}
				{\rm div}(\mathcal{L}_{\xi} \mathring{Ric})=-(n-2)\mathring{Ric}(\nabla \rho).
			\end{equation}
			\item [(iii)] \begin{eqnarray*}
				{\rm div}(\rho^m\mathcal{L}_\xi \mathring{ Ric}(\nabla \rho))
				&=&-\rho^{m-1}\big[m(n-2) \mathring{\nabla^2 \rho}(\nabla \rho,\nabla \rho)+\rho\,g({\rm div}(\mathcal{L}_\xi \mathring{ Ric}),\nabla \rho)\\
				&&+(n-2)\rho |\mathring{\nabla^2}  \rho|^2\big].
			\end{eqnarray*} Moreover, $$\int_M\rho^{m-1}\Big(m(n-2) \mathring{\nabla^2 \rho}(\nabla \rho,\nabla \rho)+\rho\,g({\rm div}(\mathcal{L}_\xi \mathring{ Ric}),\nabla \rho)+(n-2)\rho |\mathring{\nabla^2}  \rho|^2\Big) dM=0.$$
		\end{enumerate}
	\end{lemma}
	
	\begin{proof}
		For the first item we will use Lemma 3.2 due to Hwang and Yun \cite{Hwang-yun} (see also, e.g.,~\cite{yano1,yano}), which gives $$\mathcal{L}_\xi \mathring{Ric}=-(n-2)\nabla^{2}\rho+\frac{n-2}{n}\Delta \rho\, g=-(n-2)\mathring{\nabla^{2}\,\rho}.$$
		
		The first part of  item (ii) follows immediately from Bochner formula. Furthermore, if the scalar curvature $R$ is constant, then, by Lemma 2.2  in \cite{Hwang-yun},  $\rho$  satisfies 
		\begin{equation}\label{eqautofuncrho}
			\Delta \rho=-\frac{R}{n-1}\rho,
		\end{equation}  which give us
		\begin{eqnarray*}
			{\rm div}(\mathcal{L}_{\xi} \mathring{Ric})	&=& -(n-2)\mathring{Ric}(\nabla \rho).
		\end{eqnarray*}
		
		For the last item, we observe that 
			\begin{eqnarray*}
			{\rm div}(\rho^m\mathcal{L}_\xi \mathring{ Ric}(\nabla \rho))&=&m\rho^{m-1}\mathcal{L}_\xi \mathring{ Ric}(\nabla \rho,\nabla \rho)+\rho^m{\rm div}(\mathcal{L}_\xi \mathring{ Ric}(\nabla \rho))\\
			&=&m\rho^{m-1}\mathcal{L}_\xi \mathring{ Ric}(\nabla \rho,\nabla \rho)+\rho^m g({\rm div}(\mathcal{L}_\xi \mathring{ Ric}),\nabla \rho)+\rho^m g(\mathcal{L}_\xi \mathring{ Ric},\nabla^2 \rho).
			\end{eqnarray*}
			
Next we use the first item to deduce 
       \begin{eqnarray*}
			{\rm div}(\rho^m\mathcal{L}_\xi \mathring{ Ric}(\nabla \rho))&=&-m(n-2)\rho^{m-1} \mathring{\nabla^2 \rho}(\nabla \rho,\nabla \rho)+\rho^m g({\rm div}(\mathcal{L}_\xi \mathring{ Ric}),\nabla \rho)\\
			&&-(n-2)\rho^m|\mathring{\nabla^2}  \rho|^2\\
			&=&-\rho^{m-1}\big[m(n-2) \mathring{\nabla^2 \rho}(\nabla \rho,\nabla \rho)+\rho\,g({\rm div}(\mathcal{L}_\xi \mathring{ Ric}),\nabla \rho)\\
			&&+(n-2)\rho |\mathring{\nabla^2}  \rho|^2\big].
		\end{eqnarray*}
	
	So, the integral equation follows immediately by integrating equation contained in the third item. 		
	\end{proof}

The following proposition will be used in the proof of the Theorems \ref{thm1} and \ref{thm3}.

\begin{proposition}\label{propbochner} 
	Let $\xi$ be a non trivial conformal vector field with conformal factor $\rho$ on a compact Riemannian manifold $(M^n, g)$ with smooth totally geodesic boundary $\partial M$  of dimension $n \geq 3 $ and $\rho|_{\partial M}=0$. If   the scalar curvature $R$ of $(M, g)$ is constant , then 
	\begin{equation}
		\int_M \Big(\mathring{Ric}(\nabla \rho,\nabla \rho)+|\mathring{\nabla}^2 \rho|^2\Big)dM=0,
	\end{equation}
	or, equivalently, 
	\begin{equation}
		\int_M  \Big(Ric(\nabla \rho,\nabla \rho)+|\nabla^2 \rho|^2-(\Delta \rho)^2\Big)dM=0.
	\end{equation}
\end{proposition}

\begin{proof} By relation \eqref{eqautofuncrho}, we have $R>0$, since $\rho=0$ on $\partial M$. Taking the divergence of item (i) in Lemma \ref{yanod}, using Bochner formula, and \eqref{mainequation}, we get
	\begin{eqnarray*}
		div(\mathcal{L}_{\xi} \mathring{Ric}(\nabla \rho))&=&\langle\mathcal{L}_{\xi} \mathring{Ric},\nabla^2 \rho\rangle+\langle div\,(\mathcal{L}_{\xi} \mathring{Ric}),\nabla \rho\rangle\\
		&=&-(n-2)|\mathring{\nabla}^2 \rho|^2-(n-2)\mathring{Ric}(\nabla \rho,\nabla \rho). 
	\end{eqnarray*}
	
	Since $n\geq 3$, using  once more \eqref{mainequation}, integrating over $M,$ and applying the Divergence Theorem, one gets
	$$\int_{\partial M }\mathring{\nabla}^2\rho(\nabla \rho,\nu)d\sigma\,=\int_M\Big(|\mathring{\nabla}^2 \rho|^2+\mathring{Ric}(\nabla \rho,\nabla \rho)\Big)dM,$$ where $\nu$ is a unit outward normal vector field along $\partial M$.
	Nevertheless, choosing a local orthonormal frame $\{e_1,\ldots,e_n\}$, such that $e_n=\nu$, we have
	\begin{eqnarray*}
		\int_{\partial M }\mathring{\nabla}^2\rho(\nabla \rho,\nu)d\sigma\, &=&\int_{\partial M } \Big(\nabla^2 \rho(\nabla \rho,\nu)-\frac{\Delta \rho}{n} g(\nabla \rho,\nu)\Big)d\sigma\,\\
		&=& \int_{\partial M } \Big(\sum_{i=1}^{n}\rho_{i\nu}\rho_i-\frac{1}{n}\sum_{i=1}^{n}\rho_{ii} \rho_\nu\Big)d\sigma\,\\
		&=& \int_{\partial M }\Big(\sum_{i=1}^{n-1}\rho_{i\nu}\rho_i-\frac{1}{n}\sum_{i=1}^{n-1}\rho_{ii} \rho_\nu +\big(1-\frac{1}{n}\big)\rho_{\nu \nu}\rho_\nu \Big)d\sigma.
	\end{eqnarray*}
	Using that $\partial M$ is totally geodesic one gets
	$$\sum_{i=1}^{n-1}\rho_{ii}=\Delta^{\partial M}\rho \mbox{ and } \sum_{i=1}^{n-1}\rho_{i\nu}\rho_i=g(\nabla^{\partial M} \rho_\nu,\nabla^{\partial M}  \rho).$$
	
	Applying the Divergence Theorem on $\partial M$ we obtain
	\begin{eqnarray*}
		\int_{\partial M }\mathring{\nabla}^2\rho(\nabla \rho,\nu)d\sigma\, 
		&=& \int_{\partial M }\Big(g(\nabla^{\partial M} \rho_\nu,\nabla^{\partial M}  \rho)-\frac{1}{n}\rho_\nu\Delta^{\partial M}\rho+\big(1-\frac{1}{n}\big)\rho_{\nu \nu}\rho_\nu\Big)\,d\sigma\\
		&=& -\int_{\partial M }\Big(\rho\Delta^{\partial M} \rho_\nu -\frac{1}{n}\rho\,\Delta^{\partial M}\rho_\nu+\big(1-\frac{1}{n}\big)\rho_{\nu \nu}\rho_\nu\Big)\,d\sigma\\
		&=&\frac{(n-1)}{n}\int_{\partial M}\rho_{\nu \nu}\rho_\nu\,d\sigma,
	\end{eqnarray*}
	and therefore, 
	\begin{equation}\label{equaricnabl2}
		\int_M\Big(|\mathring{\nabla}^2 \rho|^2+\mathring{Ric}(\nabla \rho,\nabla \rho)\Big)dM=\frac{(n-1)}{n}\int_{\partial M}\rho_{\nu \nu}\rho_\nu\,d\sigma.
	\end{equation}

	Using relation \eqref{eqautofuncrho} in Bochner formula  and integrating over $M,$ we have
	\begin{eqnarray*}
		\dfrac{1}{2}\int_M {\rm div}(\nabla |\nabla \rho|^2)dM=\int_M \Big(|\nabla^2 \rho|^2+Ric(\nabla \rho,\nabla \rho)-\frac{R}{n-1}|\nabla \rho|^2\Big) dM.
	\end{eqnarray*}
	
	Since  ${\rm div} (\rho \nabla \rho)=-\dfrac{R}{n-1}\rho^2+|\nabla \rho|^2$ and $\rho=0$ on $\partial M$, we have
	\begin{eqnarray*}
		\frac{1}{2}\int_{\partial M} \langle\nabla |\nabla \rho|^2,\nu\rangle d\sigma=\int_M \Big(Ric(\nabla \rho,\nabla \rho)+ |\nabla^2 \rho|^2-(\Delta \rho)^2\Big)dM.
	\end{eqnarray*}
	
	Noting that $\mathring{Ric}(\nabla \rho,\nabla \rho)=Ric(\nabla \rho,\nabla \rho)-\dfrac{R}{n}|\nabla \rho|^2$, and $\displaystyle\int_M|\nabla \rho|^2dM=\frac{R}{n-1}\int_M \rho^2 dM$, we get
	\begin{eqnarray*}
		\int_{\partial M} \langle\nabla_{\nabla \rho} \nabla \rho ,\nu\rangle d\sigma=\int_M \Big(\mathring{Ric}(\nabla \rho,\nabla \rho)+ |\nabla^2 \rho|^2-\frac{1}{n}(\Delta \rho)^2\Big) dM.
	\end{eqnarray*}
	
	We already showed that
	$$\int_{\partial M} \langle\nabla_{\nabla \rho} \nabla \rho ,\nu\rangle \,d\sigma=
	\int_{\partial M }\Big(g(\nabla^{\partial M} \rho_\nu,\nabla^{\partial M}  \rho)+\rho_{\nu \nu}\rho_\nu\Big) d\sigma=\int_{\partial M }\rho_{\nu \nu}\rho_\nu \,d\sigma,$$
	hence,
	
	\begin{equation}\label{eqnabla1}
		\int_M \Big(\mathring{ Ric}(\nabla \rho,\nabla \rho)+ |\mathring{\nabla}^2 \rho|^2\Big) dM=\int_{\partial M }\rho_{\nu \nu}\rho_\nu \,d\sigma.
	\end{equation}
	
	By Equations~\eqref{equaricnabl2} and~\eqref{eqnabla1}, we achieve
$\displaystyle\int_M \Big(\mathring{ Ric}(\nabla \rho,\nabla \rho)+ |\mathring{\nabla}^2 \rho|^2\Big) dM=0,$ that finishes the proof of the proposition.
\end{proof}

Next, we present a useful   lemma  which is an important tool to prove our main results.

\begin{lemma}\label{lematensor} Let $\xi$ be a nontrivial conformal vector field on a Riemannian manifold $(M^n, g)$ with conformal factor $\rho$. Then
	\begin{enumerate}
		\item[(i)]\label{lemadiv} $
		{\rm div}\,(|\mathring{Ric}|^2\xi)=-2(n-2)\langle \nabla^2\rho,\mathring{Ric}\rangle+(n-4)\rho|\mathring{Ric}|^2.$
		\item[(ii)]\label{lemadiv2}  $
		{\rm div}\,(\rho|\mathring{Ric}|^2\xi)=|\mathring{Ric}|^2\langle\nabla \rho,\xi\rangle -2(n-2)\rho\langle \nabla^2 \rho,\mathring{Ric}\rangle+(n-4)\rho^2|\mathring{Ric}|^2.$
		\item [(iii)] \label{lemadiv3} $n\,{\rm div}\,(|\mathring{Ric}|^2\xi+2(n-2)\mathring{Ric}(\nabla \rho))={(n-2)^2}\langle \nabla R,\nabla \rho\rangle+n(n-4)\rho|\mathring{Ric}|^2.$
	\end{enumerate}
	
\end{lemma}
\begin{proof}
	
	In fact, firstly we note that
	$${\rm div}\,(|\mathring{Ric}|^2 \xi)=\langle \nabla|\mathring{Ric}|^2,\xi\rangle+ n\rho|\mathring{Ric}|^2.$$ By Yano \cite{yano}, we have
	$$\mathcal{L}_{\xi}|\mathring{Ric}|^2=-2(n-2)\langle\nabla^2 \rho,\mathring{Ric} \rangle-4\rho |\mathring{Ric}|^2,$$ and using that $\mathcal{L}_{\xi}|\mathring{Ric}|^2=\langle \nabla|\mathring{Ric}|^2,\xi\rangle$, we obtain
	$${\rm div}\,(|\mathring{Ric}|^2\xi)=-2(n-2)\langle \nabla^2 \rho,\mathring{Ric}\rangle+(n-4)\rho|\mathring{Ric}|,$$
	which establishes (i). 
	
	On the other hand, from the first item we obtain 	$${\rm div}\,(\rho|\mathring{Ric}|^2\xi)=|\mathring{Ric}|^2\langle\nabla \rho,\xi\rangle-2(n-2)\rho\langle \nabla^2 \rho,\mathring{Ric}\rangle+(n-4)\rho^2|\mathring{Ric}|^2.$$
	Now using that  ${\rm div}(\mathring{Ric}(\nabla \rho))=({\rm div}\,\mathring{Ric})(\nabla \rho)+\langle \mathring{Ric},\nabla^2\rho\rangle,$ and ${\rm div}\,\mathring{Ric}=\dfrac{n-2}{2n}\nabla R,$ we have
	\begin{eqnarray*}
		\langle \mathring{Ric},\nabla^2\rho\rangle&=&	{\rm div}(\mathring{Ric}(\nabla \rho))-\dfrac{n-2}{2n}\langle \nabla R,\nabla \rho\rangle,
	\end{eqnarray*}and, by the first item, again, we infer
	$$
	{\rm div}\,(|\mathring{Ric}|^2\xi)=-2(n-2)\,{\rm div}(\mathring{Ric}(\nabla \rho))+\dfrac{(n-2)^2}{n}\langle \nabla R,\nabla \rho\rangle+(n-4)\rho|\mathring{Ric}|^2,$$ which implies that
	\begin{equation}
		n\,{\rm div}\,(|\mathring{Ric}|^2\xi+2(n-2)\mathring{Ric}(\nabla \rho))={(n-2)^2}\langle \nabla R,\nabla \rho\rangle+n(n-4)\rho|\mathring{Ric}|^2.
    \end{equation}Thus we finish the proof of the lemma.
\end{proof}

Now we present some properties  when we have a 4-dimensional manifold. 

\begin{proposition}\label{propdiv}
	Let $(M^4,g)$ be a Riemannian manifold with constant scalar curvature and $\xi$ be a nontrivial conformal vector field on $M^4$ with conformal factor $\rho$. If $f:M^4\to \mathbb{R}$ is a smooth function, then
	\begin{equation}\label{eqprodiv}
		{\rm div}\Big(f^m \big( |\mathring{Ric}|^2\xi+4 \mathring{Ric}(\nabla \rho)\big)\Big)=|\mathring{Ric}|^2\langle \xi,\nabla f^m \rangle+4\mathring{Ric}(\nabla \rho,\nabla f^m).
	\end{equation}
\end{proposition}

\begin{proof}
	In fact, by item (iii) of Lemma~\ref{lematensor}, since the scalar curvature is constant, and $n=4$, we have ${\rm div}( |\mathring{Ric}|^2\xi+4\mathring{Ric}(\nabla \rho))=0.$
	Hence,
	\begin{eqnarray*}
		{\rm div}\Big(f^m \big(|\mathring{Ric}|^2\xi+4\mathring{Ric}(\nabla \rho)\big)\Big)&=&\langle |\mathring{Ric}|^2\xi+4\mathring{Ric}(\nabla \rho),\nabla f^m\rangle\\
		&=&|\mathring{Ric}|^2\langle \xi,\nabla f^m \rangle+4\mathring{Ric}(\nabla \rho,\nabla f^m).
	\end{eqnarray*}
	\end{proof}
	
	Consequently, we have the following corollary.
	
\begin{corollary}\label{Mcorollary}
	Let $(M^4, g)$ be a Riemannian manifold with constant scalar curvature $R$, and let $\xi$ be a non trivial conformal vector field on $M^4$ with conformal factor $\rho$. If $M^4$ is either closed, or, compact with smooth boundary $\partial M$, where $\rho = 0$ on $\partial M$, then the following holds: $$\int_M \Big(|\mathring{Ric}|^2\langle \xi,\nabla \rho \rangle+4\mathring{Ric}(\nabla \rho,\nabla \rho)\Big)dM=0.$$
\end{corollary}

\section{Proof of Theorems} \label{prooftheorems}

In this section, we will present the proofs of the theorems.

\subsection{Proof of Theorem \ref{thm1}}
Since $\mathring{Ric}(\nabla \rho)=0,$ by hypothesis, using 
Proposition \ref{propbochner}, we deduce 
$$\int_M |\mathring{\nabla} ^2 \rho|^2dM=0.$$

Therefore, by the hypothesis on $\rho$, we can apply [ \cite{reilly}, Theorem B] to ensure that $M$ is isometric to a geodesic ball on $\mathbb{S}^n(c)$. Since $\partial M$ is totally geodesic, we conclude that $M$ is isometric to a hemisphere  $\mathbb{S}^n_+(c)$.

\subsection{Proof of Theorem \ref{thm2}}

Using Bochner formula one gets 
$$\int_M \Big(|\nabla^2 \rho|^2+Ric(\nabla \rho,\nabla \rho)+\langle\nabla \Delta \rho,\nabla \rho \rangle\Big) dM=0.$$
Since  $R$ is constant, we have  $\Delta \rho=-\dfrac{R}{n-1}\rho$. Then,
$$\int_M \Big(|\nabla^2 \rho|^2+Ric(\nabla \rho,\nabla \rho)-\frac{R}{n-1}|\nabla \rho|^2\Big) dM=0.$$
Using once more that $\rho$ is an eigenfunction, we stablish that $\displaystyle\int_M |\nabla \rho|^2dM=\frac{R}{n-1}\int \rho^2 dM$. 

On the other hand, $\mathring{Ric}(\nabla \rho,\nabla \rho)=Ric(\nabla \rho,\nabla \rho)-\dfrac{R}{n}|\nabla \rho|^2$ and 
$|\mathring{\nabla}^2 \rho|^2=|\nabla^2 \rho|^2-\dfrac{1}{n}(\Delta \rho )^2$, give us   
$$\int_M \Big(\mathring{Ric}(\nabla \rho,\nabla \rho)+|\mathring{\nabla^2} \rho|^2\Big)dM=0.$$
Using  Corollary \ref{Mcorollary}, we  finally get
$$\int_M |\mathring{Ric}|^2\langle \xi,\nabla \rho \rangle dM =4\int |\mathring{\nabla}^2 \rho|^2dM.$$

Therefore, by the hypothesis we assume that $\displaystyle\int_M |\mathring{Ric}|^2\langle \xi,\nabla \rho \rangle dM\leq0.$ Whence we  can use  Obata’s result \cite{obata} to  conclude that $(M^4, g)$ is isometric to a standard sphere $\mathbb{S}^4(c)$. 

\subsection{Proof of Theorem \ref{thm3}}

In fact, by Proposition \ref{propbochner}, we have 
$$\int_M \Big(\mathring{Ric}(\nabla \rho,\nabla \rho)+|\mathring{\nabla}^2 \rho|^2\Big)dM=0.$$
Combining with Corollary \ref{Mcorollary}, we get 
$$\int_M |\mathring{Ric}|^2\langle \xi,\nabla \rho \rangle dM =4\int |\mathring{\nabla}^2 \rho|^2dM.$$

Therefore, by the hypothesis on $\rho$, we can apply [ \cite{reilly}, Theorem B] to conclude that $M$ is isometric to a geodesic ball on $\mathbb{S}^4(c)$. Since $\partial M$ is totally geodesic, we conclude that $M$ is isometric to a hemisphere  $\mathbb{S}^4_+(c)$.

\subsection{Proof of Theorem \ref{thm4}} 

Using the relation \eqref{5} and item (ii)  in Lemma \ref{lematensor},  we obtain
$${\rm div}\,(\rho|\mathring{Ric}|^2\xi)=|\mathring{Ric}|^2\langle\nabla \rho,\xi\rangle-n|\mathring{\nabla^2 \rho}|^2.$$
Integrating over $M$ and using the Divergence Theorem, one gets
\begin{equation}\label{abequation}
	\int_M |\mathring{Ric}|^2\langle\nabla \rho,\xi\rangle dM=n\int _M|\mathring{\nabla^2 \rho}|^2 dM.
\end{equation}

Hence, by Equations~\eqref{5} and~\eqref{abequation}, we have 
$$\int_M |\mathring{Ric}|^2 \psi\,dM = 0,$$ where $\psi= n\rho^2-\langle \nabla \rho, \xi\rangle$ as defined in the theorem. Since $\psi\neq 0$, we obtain  $| \mathring{Ric}|^2=0$.

Again, by relation \eqref{5} and the definition of $\mathring{\nabla^{2}\,\rho},$ yield
\begin{equation}\label{typeobataequation}
	\nabla^2 \rho =\frac{\Delta \rho}{n}g.
\end{equation}
Since $\Delta \rho=-\dfrac{R}{n-1}\rho$ and $\rho=0$ on $\partial M$,  it is not difficult to prove that $R>0$. Then,
$$\nabla^2 \rho =-\frac{R \rho}{n(n-1)}g.$$

Therefore, by the hypothesis on $\rho$, we can apply [ \cite{reilly}, Theorem B] to guarantee that $M$ is isometric to a geodesic ball on $\mathbb{S}^n(c)$. Since $\partial M$ is totally geodesic, we conclude that $M$ is isometric to a hemisphere  $\mathbb{S}^n_+(c)$.

\end{document}